\documentclass[11pt,a4paper]{amsart}

 \usepackage[utf8x]{inputenc}
\usepackage{latexsym}
\usepackage{color,graphicx,shortvrb}
\usepackage{amsmath, amssymb}
\usepackage{amsfonts}
\usepackage[colorlinks, bookmarks=true]{hyperref}

\newtheorem{theorem}{Theorem}[section]
\newtheorem{lemma}[theorem]{Lemma}

\theoremstyle{definition}

\newtheorem{remark}[theorem]{Remark}

\author{J. M. Almira}
\title{On Loewner's characterization of polynomials}

\begin{document}



\address{Departamento de Matem\'{a}ticas, Universidad de Ja\'{e}n, E.P.S. Linares,  Campus Cient\'{\i}fico Tecnol\'{o}gico de Linares, Cintur\'{o}n Sur s/n, 23700 Linares, Spain}
\email{jmalmira@ujaen.es}


\begin{abstract}
We give a new demonstration of Loewner's characterization of polynomials, solving in the positive a conjecture proposed by Laird and McCann in 1984.
\end{abstract}

\maketitle

\markboth{J. M. Almira}{ Loewner's characterization of polynomials}

\section{Introduction}

We give a new proof of the following result:

\begin{theorem}[Loewner, 1959]
Assume that $d>1$ is a natural number. Let $f\in C(\mathbb{R}^d)$ and let $R_f=\mathbf{span}\{f(Lx):L \text{ is an isometry of }\mathbb{R}^d\}$. Then $f$ is an ordinary polynomial if and only if $\dim R_f<\infty$. 
\end{theorem}

This theorem was demonstrated by Loewner using Anselone-Koreevar's theorem \cite{anselone}, which claims that exponential polynomials can be characterized as elements of finite dimensional translation invariant subspaces of $C(\mathbb{R}^d)$. Later on, in 1984, Laird and McCann \cite{Laird} gave a characterization of polynomials as the elements of finite dimensional subspaces of $C(\mathbb{R}^d)$ which are simultaneously translation and dilation invariant (their result allows $d=1$). Their proof used Fr\'{e}chet's characterization of polynomials, instead of Anselone-Koreevar's theorem. Then they conjectured that a similar proof for Loewner's theorem should exist. The main goal of this short note is to solve their conjecture in the positive. 

\section{Main result}
Note that, if $\dim R_f<\infty$ then for all $y\in\mathbb{R}^d$ and all $P\in \mathbf{O}(d):=\{A\in\mathbf{GL}_d(\mathbb{R}): A^t=A^{-1}\}$  the maps $\tau_y(g)(x)=g(x+y)$ and $O_P(g)(x)=g(Px)$ are well defined as operators from $R_f$ into $R_f$. In fact, they are automorphisms of this space, since they are injective and $\dim R_f<\infty$. Indeed, $(\tau_y)^{-1}=\tau_{-y}$ and $(O_P)^{-1}=O_{P^{-1}}=O_{P^{t}}$.

Let  $X_d$ denote indistinctly either the space $\mathcal{D}(\mathbb{R}^d)'$ of Schwartz complex valued distributions  defined on $\mathbb{R}^d$ or the space $C(\mathbb{R}^d)$ of continuous complex valued functions defined on $\mathbb{R}^d$. For $f\in \mathcal{D}(\mathbb{R}^d)'$  we can introduce the translation operator
\[
\tau_h (f)\{\phi\} =f\{\tau_{-h}(\phi)\} ,\text{ where } h\in\mathbb{R}^d \text{ and } \phi \in\mathcal{D}(\mathbb{R}^d) \text{ is any test function,}
\]
and the operator
\[
O_P(f)\{\phi\} =\frac{1}{|\det(P)|}f\{O_{P^{-1}}(\phi)\}, 
\] 
where  $P\in\mathbf{GL}_d(\mathbb{E})$  is any invertible matrix,   $\phi \in\mathcal{D}(\mathbb{R}^d)$   is any test function, and 
$O_{P^{-1}}(\phi)(x)=\phi(P^{-1}x)$ for all $x\in\mathbb{R}^d$.

We demonstrate the following result, which is stronger  than Loewner's theorem, since it works for distributions.

\begin{theorem} \label{main} Let $d\geq 2$ be a natural number, let $f\in X_d$ and assume that, for a certain finite dimensional space $V\subseteq X_d$ we have that:
\begin{itemize}
\item $f\in V$
\item $V$ is translation invariant (i.e., $\tau_h (V)\subseteq V$ for all $h\in \mathbb{R}^d$).
\item $V$ is invariant by orthogonal transformations of $\mathbb{R}^d$ (i.e., $O_{P} (V) \subseteq  V$, for all $P\in\mathbf{O}(d)$). 
\end{itemize}
Then $f$ is, in distributional sense, an ordinary polynomial on $\mathbb{R}^d$. In particular, $f$ is equal almost everywhere to an ordinary polynomial and, if $f$ is a continuous ordinary function, then it is an ordinary polynomial.  
\end{theorem}

Our proof is based on the following theorem

\begin{theorem}\label{pre} Let $f$ be a complex valued distribution defined on $\mathbb{R}^d$. Assume that $q(\tau_y)f=0$ for all $y$ with $\|y\|\leq \delta$, for a certain polynomial $q(z)=a_0+a_1z+\cdots a_nz^n$ such that $a_0\neq 0$. Then $f$ is, in distributional sense, an ordinary polynomial.
\end{theorem}

\begin{proof} 
 By assumption, $q(\tau_y)f=0$  for all $y$ with $\|y\|\leq \delta$, which means that 
\begin{equation}\label{uno}
0=\sum_{k=0}^na_k\tau_{ky}f(x) \text{ for all } y\in B_{d}(\delta):=\{h\in\mathbb{R}^d:\|h\|<\delta\} .
\end{equation}
Assume that $y,h_1\in  B_{d}(\delta/2)$ (so that $y^*=y-h_1\in B_d(\delta)$) and use \eqref{uno} with $y^*$ to conclude that:
\begin{equation}\label{unonuevo}
0=\sum_{k=0}^na_k\tau_{ky^*}f(x) = \sum_{k=0}^na_k\tau_{ky-kh_1}f(x) \text{ for all } y,h_1 \in B_{d}(\delta/2) .
\end{equation}

Apply $\tau_{nh_1}$ to both sides of the equation. Then 
\begin{equation}\label{tres}
0= \sum_{k=0}^na_k\tau_{nh_1}\tau_{ky-kh_1}f(x) = \sum_{k=0}^{n}a_k\tau_{(n-k)h_1}(\tau_{ky}f)(x) \text{ for all } y,h_1\in B_{d}(\delta/2).
\end{equation}
Taking differences between \eqref{tres} and \eqref{uno}, we conclude that
\begin{equation}\label{cuatro}
0= \sum_{k=0}^{n-1}a_k\Delta_{(n-k)h_1}(\tau_{ky}f)(x) \text{ for all } y,h_1\in B_{d}(\delta/2).
\end{equation}
We can repeat the argument, reducing the norm of $y, h_1$ to $\delta/4$, $\delta/8$, etc., which leads to the equation
\[
a_0\Delta_{h_n} \Delta_{2h_{n-1}} \cdots  \Delta_{(n-1)h_2} \Delta_{nh_1}(f)(x) =0 \text{ for all } h_1,\cdots,h_n\in B_{d}(\delta/2^n).
\] 
The result follows from Montel's type version of  Fr\'{e}chet's theorem for distributions, since $a_0\neq 0$ by hyphotesis (see, e.g., \cite{A1,A2,A3,A4}). 
\end{proof}

We also use the following technical result, which is well  known (see, for example, \cite{La1, McCoy, RaRo}  and \cite{Fro}, for the original exposition of this result):

\begin{lemma}[Frobenius, 1896]
Let $V$ be a finite dimensional complex vector space. Assume that $T,S:V\to V$ are commuting linear operators (i.e., $TS=ST$). Then they are simultaneously triangularizable. In particular, the eigenvalues $\lambda_i(T), \lambda_i(S)$ and $\lambda_i(TS)$ of $T,S$ and $TS$, respectively, can be arranged, counting multiplicities,  in such a way that 
$\lambda_i(TS)=\lambda_i(T)\lambda_i(S)$ for $i=1,\cdots, \dim V$. 
\end{lemma}

\begin{proof}[Proof of Theorem \ref{main}] Obviously
\[
\tau_{y+z}=\tau_y\tau_z=\tau_z\tau_y \text{ for all } y,z\in \mathbb{R}^d,
\] 
and  
\begin{equation} \label{semejanza}
\tau_{Py}=(O_{P})^{-1}\tau_y O_{P} = O_{P^t}\tau_y O_{P} \text{ for all } y\in\mathbb{R}^d.
\end{equation}
To demonstrate \eqref{semejanza} we consider first the case of ordinary functions. In that setting it is clear that  $(O_{P})^{-1}=O_{(P)^{-1}}=O_{P^t}$ and 
\begin{eqnarray*}
(O_{P^{-1}}\tau_y O_{P}g)(x) &=&  (O_{P^{-1}}\tau_y (g(Px)) =  O_{P^{-1}}(g(P(x+y))) \\
&=& O_{P^{-1}}(g(Px+Py))= g(P^{-1}Px+Py)\\
&=& g(x+Py)=\tau_{Py}(g)(x),
\end{eqnarray*}
for every function  $g$. 
Assume now that $V$ is a space of distributions and $g\in V$. Then, for every test function $\phi$, 
\begin{eqnarray*}
(O_{P^{-1}}\tau_y O_P)(g)\{\phi\} &=& O_{P^{-1}}(\tau_y O_P(g))\{\phi\}  = \frac{1}{|\det(P^{-1})|} (\tau_y O_P(g))\{O_P(\phi)\} \\
 &= & \frac{1}{|\det(P^{-1})|} \tau_y (O_P(g))\{O_P(\phi)\} = \frac{1}{|\det(P^{-1})|} (O_P(g))\{\tau_{-y}O_P(\phi)\} \\
 &=& \frac{1}{|\det(P^{-1})|}  \frac{1}{|\det(P)|} g\{O_{P^{-1}}\tau_{-y}O_P(\phi)\} =  g\{O_{P^{-1}}\tau_{-y}O_P(\phi)\} \\
 &=&  g\{\tau_{-Py}(\phi)\}  = \tau_{Py}(g)\{\phi\}, \\
\end{eqnarray*}
which is what we wanted to prove.  It follows that  
\begin{itemize}
\item[$(a)$] For any $y\in\mathbb{R}^d$, the operators $\tau_{y}:V\to V$ and $\tau_{Py}:V\to V$ have the very same eigenvalues and characteristic polynomial. 

\item[$(b)$]  For any $y,z\in\mathbb{R}^d$, the  operators $\tau_{y}:V\to V$ and $\tau_z:V\to V$ are simultaneously triangularizable, since they are commuting operators (i.e., 
$\tau_y\tau_z=\tau_z\tau_y$). In particular, it is possible to arrange the eigenvalues   $\lambda_i(\tau_y), \lambda_i(\tau_z)$ and $\lambda_i(\tau_{y+z})$ of  $\tau_{y}$, $\tau_z$ and $\tau_{y+z}$, respectively,  in such a way that
\[
\lambda_i(\tau_{y+z})=\lambda_i(\tau_y)\lambda_i(\tau_z), \ \ i=1,\cdots, N:=\dim V.
\]
\end{itemize}
Let $y\in\mathbb{R}^d$ . From $(a)$ we have that, for a certain permutation $\sigma$ of $\{1,\cdots, N\}$ (which depends on $P$),  
\[
\lambda_i(\tau_{Py})=\lambda_{\sigma(i)}(\tau_y), \ i=1,\cdots, N.
\] 
Moreover, if we set $z=(P-I)y$, then $Py=y+(P-I)y)$ and $(b)$ implies that:
\[
\lambda_{\sigma(i)}(\tau_y)=\lambda_i(\tau_{Py})=\lambda_i(\tau_y)\lambda_i(\tau_{(P-I)y}), \ i=1,\cdots,N.
\] 
Hence 
\[
\lambda_i(\tau_{(P-I)y})=\frac{\lambda_{\sigma(i)}(\tau_y)}{\lambda_i(\tau_y)}, \text{  } i=1,\cdots,N,
\]
since $\tau_y$ is an automorphism, which implies that $\lambda_i(\tau_y)\neq 0$ for all $i$. 

Take $y=e_1=(1,0,\cdots, 0)\in\mathbb{R}^d$. Let us consider the polynomial 
$$q(z)=\prod_{\sigma\in S_N}\prod_{i=1}^N(z- \frac{\lambda_{\sigma(i)}(\tau_{e_1})}{\lambda_i(\tau_{e_1})}),$$ 
where $S_N$ denotes the set of permutations of $\{1,\cdots, N\}$.  Then $q(z)=a_0+a_1z+\cdots a_mz^m$ for certain coefficients $a_k$, with $m=N\cdot N!$, and $a_0\neq 0$ because all eigenvalues $\lambda_i(\tau_{e_1})$ are different from $0$ (since $\tau_{e_1}$ is injective). Moreover, $q(z)$  is a multiple of the characteristic polynomial of the operator $\tau_{(P-I)e_1}$, for every $P\in \mathbf{O}(d)$. It follows that $q(z)$ is a multiple of the characteristic polynomial of $\tau_z$ for every $z\in\mathbb{R}^d$ whose norm is equal to the norm of $(P-I)e_1$ for some $P\in \mathbf{O}(d)$, since $\|z_1\|=\|z_2\|$ implies that there exist $P\in \mathbf{O}(d)$ such that $Pz_1=z_2$, and $\tau_{z_1}$, $\tau_{Pz_1}$ have the very same characteristic polynomial.  Obviously, $\{\|(P-I)e_1\|: P\in\mathbf{O}(d)\}=[0,2]$, since $d>1$. Henceforth, 
Hamilton's theorem implies that $q(\tau_{z})=0$ for all $\|z\|\leq 2$.   Now we can apply Theorem \ref{pre} to conclude that $f$ is, in distributional sense, an ordinary polynomial. 
\end{proof}

\begin{remark} For $d=1$ the result is false, since $V=\mathbf{span}\{2^{x},2^{-x}\}$ is invariant by isometries of $\mathbb{R}$. 
\end{remark}

\bibliographystyle{amsplain}

\begin{thebibliography}{99}

\bibitem{A1} {\sc A. Aksoy, J. M. Almira, } On Montel and Montel-Popoviciu Theorems in several variables, Aequationes Mathematicae,  \textbf{89} (2015) 1335-1357.

\bibitem{A2} {\sc J. M. Almira, }  Montel's theorem and subspaces of distributions which are $\Delta^m$-invariant, Numer. Functional Anal. Optimiz. \textbf{35} (4) (2014) 389--403.

\bibitem{A3} {\sc J. M. Almira, K. F. Abu-Helaiel, } On Montel's theorem in several variables, Carpathian Journal of Mathematics,  \textbf{31} (2015), 1--10.

\bibitem{A4} {\sc J. M. Almira, L. Sz\'{e}kelyhidi, } Montel-type theorems for exponential polynomials, Dem. Math. \textbf{49} (2) (2016) 197-212.

\bibitem{anselone} {\sc P. M. Anselone, J. Korevaar, } Translation invariant subspaces of finite dimension, \emph{Proc. Amer. Math. Soc. } \textbf{15} (1964), 747-752.

\bibitem{Fro} {\sc G. Frobenius, } Uber vertauschbare Matrizen, Sitzungsber. Akad. Wiss. Berlin 26  (1896) 601--614.


\bibitem{La1} {\sc T. J. Lalfey, } Simultaneous reduction of sets of matrices under similarity, Linear
Algebra Appl. 84 (1986) 123--138. 


\bibitem{Laird} {\sc P.G. Laird, R. McCann, } On some characterizations of polynomials,   Amer. Math. Monthly 91 (2) (1984) 114-116. 

\bibitem{Laird1} {\sc P.G. Laird,  } On  characterizations of exponential polynomials,  Pacific J.  Math. 80  (1979) 503-507. 

\bibitem{Lo} {\sc C. Loewner, } On some transformation semigroups invariant under Euclidean and non-Euclidean isometries, J. Math. Mech.  8 (1959) 393-409. 


\bibitem{McCoy} {\sc N. H. McCoy, } On the characteristic roots of matrix polynomials, Bull. Amer. Math. Sot. 42 (1936) 592--600.


\bibitem{RaRo} {\sc H. Radjavi,  P. Rosenthal ,}  \emph{Simultaneous Triangularization}, Springer-Verlag New York, Inc., 2000. 

\end{thebibliography}


\end{document}